\documentclass[a4paper,10pt]{article}
\usepackage{latexsym,amssymb,amsfonts,amsmath,amsthm,amstext,color,graphicx,times}
\usepackage[mathscr]{euscript}
\usepackage{indentfirst}
\usepackage{cite}
\usepackage{mathrsfs}
\usepackage{subcaption}

\usepackage{amsxtra}
\usepackage{dsfont}
\usepackage{stmaryrd}

\newcommand{\R}{\mathbb{R}}

\newcommand{\N}{\mathbb{N}}

\usepackage{wrapfig}
\usepackage[all]{xy}
\usepackage{tikz}
\usepackage{hyperref}
\usepackage{animate}
\usetikzlibrary{patterns}

\theoremstyle{plain}
\newtheorem{theorem}{Theorem}[section]

\newtheorem{proposition}[theorem]{Proposition}
\newtheorem{corollary}[theorem]{Corollary}

\theoremstyle{definition}

\newtheorem{example}[theorem]{Example}

\theoremstyle{remark}
\newtheorem{remark}[theorem]{Remark}

\newcommand{\tos}{\rightrightarrows} 
\DeclareMathOperator*{\argmin}{arg\,min}

\DeclareMathOperator*{\co}{co}

\DeclareMathOperator*{\dist}{dist}

\DeclareMathOperator*{\nep}{NEP}
\DeclareMathOperator*{\gnep}{GNEP}
\DeclareMathOperator*{\Rgnep}{RGNEP}

\DeclareMathOperator{\gra}{gra}

\title{The generalized Nash game proposed by Rosen}

\author{ 
Carlos Calder\'on 
\thanks{Instituto de Matem\'atica y Ciencias Afines. Lima, Per\'u. Email: carlos.calderon@imca.edu.pe}
\and
John Cotrina  
\thanks{Universidad del Pac\'ifico, Lima, Per\'u. Email: cotrina\_je@up.edu.pe} 
}

\begin{document}

\maketitle

\begin{abstract}
We deal with the generalized Nash game proposed by Rosen, which is a game with strategy sets that are coupled across players through a shared constraint. A reduction to a classical game is shown, and as a consequence, Rosen's result can be deduced from the one given by Arrow and Debreu. We also establish necessary and sufficient conditions for a point to be a generalized Nash equilibrium employing the variational inequality approach.
Finally, some existence results are given in the non-compact case under coerciveness conditions.
\bigskip


\bigskip


\end{abstract}

\section{Introduction}
The Nash equilibrium problem (NEP in short) \cite{Nash} consists of a finite number of players, each player has a strategy set and an objective function depending not only on his/her decision but also on the decision of his/her rival players.  Arrow and Debreu \cite{Arrow-Debreu}  considered a kind of game in which the strategy set of each player also depends on the decision of his/her rival players, they called it \emph{Abstract Economy}. Nowadays, these games are called the generalized Nash equilibrium problem (GNEP in short), see \cite{Facchinei2007}. In 1965, Rosen \cite{Rosen} dealt with a particular generalized Nash equilibrium problem, where the strategy sets are coupled across players through a shared constraint, we denote it by RGNEP. Recently, more and more researchers are interested in the RGNEP because it models real problems such as  electricity markets, environmental games, and bilateral exchanges of bads, see for instance \cite{K-2005,CK-2015,Vardar2019,K-2016}.

Rosen \cite{Rosen} established an existence result for the RGNEP under continuity and convexity assumptions (see Theorem \ref{TR}), which is not a direct consequence of the one given by Arrow and Debreu \cite{Arrow-Debreu} (see Theorem \ref{A-D} and Example \ref{example1}).
However, we will reduce the RGNEP to a classical Nash game with two players. 

On the other hand, Facchinei \emph{et al.}  \cite{FACCHINEI2007159}  extended Rosen's result by considering pseudo-convexity (see Theorem \ref{FFP}) instead of the convexity of each objective function. In this work \cite{FACCHINEI2007159}, the authors reduced the RGNEP to a Stampacchia variational inequality problem. After that, Aussel and Dutta \cite{Aussel-Dutta} presented an existence result using semi strict quasi-convexity and continuity (see Theorem \ref{AD08}), which extends results given in \cite{Rosen,FACCHINEI2007159}. The authors in \cite{Aussel-Dutta} also reduced the RGNEP to a variational inequality problem using the adjusted normal cone \cite{AuHad2005}. In the same line, recently, Bueno \emph{et al.} \cite{BCC21} dealt with the quasi-convex case (see Theorem \ref{main-result}). Thus, we will present another existence result under quasi-convexity and pseudo-continuity, which is equivalent to the one given in \cite{BCC21}. Moreover, we will show the strong relationship with the one given by Arrow and Debreu \cite{Arrow-Debreu}.

Cavazzuti \emph{et al.} \cite{Cavazzuti} also dealt with the RGNEP, where by means of the Minty variational inequality, they established sufficient and necessary conditions for a point to be a generalized equilibrium, under differentiability assumption. Thus, we will extend the result given in \cite{Cavazzuti} to the continuous case, using for that normal cones instead of the gradient. 

Recently, the case of unbounded (hence, non-compact) sets was recently dealt in the GNEP under certain coerciveness condition, see \cite{AS-2017,CHS-2020,CZ-2018}. Motivated by these works we focus in the GNEP proposed by Rosen and obtain certain existence results.

The remainder of the paper is organized as follows. 
In Section \ref{preliminaries}, we give definitions of pseudo-continuity and generalized convexity for functions, and continuity for set-valued maps. Moreover, we present some results concerning optimization problems. 
In Section \ref{motivation}, we present the generalized Nash game proposed by Rosen
and show that Rosen's theorem is a consequence of the one given by Arrow and Debreu.
In Section \ref{variational approach}, we show the existence of generalized Nash equilibria for the RGNEP with discontinuous functions and the equivalence between some existence results. Furthermore, we also establish sufficient and necessary conditions for a point to be a generalized Nash equilibrium.
Finally, in Section \ref{coerciveness}, we introduce some coerciveness conditions and obtain existence results for the RGNEP.

\section{Definitions, notations and preliminary results} \label{preliminaries}
We first recall the notion of convexity and generalized convexity. A real-valued function $f:\R^n\to\R$ is said to be:
\begin{itemize}
\item \emph{convex} if, for all $x,y\in\R^n$ and all $t\in[0,1]$
\[
f(tx+(1-t)y)\leq tf(x)+(1-t)f(y);
\]
\item \emph{quasi-convex} if, for all $x,y\in\R^n$ and all $t\in[0,1]$ 
\[
f(tx+(1-t)y)\leq \max\{f(x),f(y)\};
\]
\item \emph{semi strictly quasi-convex} if, it is quasi-convex and moreover for all $x,y\in\R^n$ such that
$f(x)<f(y)$, we have for all $t\in]0,1[$
\[
f(tx+(1-t)y)<f(y);
\]
\item \emph{pseudo-convex} if, it is differentiable and the following implication holds
\[
\langle \nabla f(x),y-x\rangle\geq0~\Rightarrow~f(y)\geq f(x).
\]
\end{itemize}
Any convex function is semi-strictly quasi-convex, which in turn is quasi-convex.
It is also clear that any convex 
and differentiable function is pseudo-convex, and this is quasi-convex. 

\,

We also recall the notion of pseudo-continuity for functions.
A real-valued function $f:\R^n\to\R$ is said to be:
\begin{itemize}
\item \emph{upper pseudo-continuous} if, for any $x,y\in X$  such that $f(x)<f(y)$, there exists a neighbourhood $V_x$ of $x$ satisfying
\[
f(x')<f(y),\mbox{ for all }x'\in V_x.
\]
\item \emph{lower pseudo-continuous} if, $-f$ is upper pseudo-continuous;
\item \emph{pseudo-continuous} if, it is lower and upper pseudo-continuous.
\end{itemize}
It is important to notice that any upper semi-continuous function is upper pseudo-continuous, but the converse is not true in general, see \cite{Co21} and its references for more details on pseudo-continuity.  

The following result is Theorem 3.2 in \cite{Scalzo}.
\begin{proposition}[Scalzo]\label{pseudo-CC}
Let $X$ be a connected topological space and $f:X\to\R$ be a function. Then $f$ is pseudo-continuous if, and only if, there exists a continuous function $u:X\to\R$ and an increasing function $h:u(X)\to\R$ such that
\[
f=h\circ u.
\]
\end{proposition}

Associated to a real-valued function $f:\R^n\to\R$ and $x\in \R^n$ we consider the following sets
\[
S^<_f(x):=\{y\in\R^n:~f(y)<f(x)\}\mbox{ and }S_f(x):=\{y\in\R^n:~f(y)\leq f(x)\}.
\]
These sets are called the strict lower level set and the lower level set of $f$ at $x$, respectively. Additionally, we also consider the \emph{adjusted level set} of $f$ at $x$
\[
S_f^a(x):=\left\lbrace \begin{matrix}
S_f(x)\cap \overline{B}(S^<_f(x),\rho_x),&x\notin \argmin f\\
S_f(x),&\mbox{otherwise}
\end{matrix}\right.
\] 
where $\rho_x:=\dist(x,S^<_f(x))$. 
The adjusted level set was introduced by Aussel and Hadjisavvas in \cite{AuHad2005}. They characterized the quasi-convexity utilizing the convexity of its adjusted level sets.

It is known that a real-valued function $f:\R^n\to\R$ is quasi-convex (resp. lower pseudo-continuous) if, and only if,  $S_f(x)$ is convex (resp. closed), for all $x\in \R^n$. 
To know more about quasi-convex functions and quasi-convex optimization, we suggest to see \cite{Aussel-book}.

\,

For each of the levels we consider their respective associated normal cone, that is:

\[
N^s(x):=\left\lbrace
\begin{matrix}
\{x^*\in\R^n:~ \langle x^*,y-x\rangle\leq0,\mbox{ for all }y\in S^<_f(x)\},& S^<_f(x)\neq\emptyset\\
\R^n,&\mbox{otherwise}
\end{matrix}
\right.
\]
\[
N(x):=\{x^*\in\R^n:~ \langle x^*,y-x\rangle\leq0,\mbox{ for all }y\in S_f(x)\}\mbox{ and }
\]
\[
N^a(x):=\{x^*\in\R^n:~ \langle x^*,y-x\rangle\leq0,\mbox{ for all }y\in S^a_f(x)\}.
\]
Since $S_f^<(x)\subset S_f^a(x)\subset S_f(x)$ we have $N(x)\subset N^a(x)\subset N^s(x)$. Moreover, the sets $N(x), N^a(x)$ and $N^s(x)$ are convex cones, closed and non-empty. Interesting properties related to the adjusted normal cone were proved in \cite{AuHad2005,BueCot16}. 

\,

The following result establishes a necessary condition to guarantee that a point is a minimizer of a real-valued function. This result is inspired by Lemma 2.1 in \cite{Cavazzuti} and from a remark given in \cite{BueCot16} on the adjusted normal cone.

\begin{proposition}\label{A}
Let $X$ be a non-empty subset of $\R^n$ and $f:\R^n\to\R$ be a function.
 If $\hat{x}\in\argmin_X f$, then
\begin{align}\label{A1}
\langle y^*,y-\hat{x}\rangle\geq0 \mbox{ for all }y\in X\mbox{ and all }y^*\in N(y).
\end{align}
\end{proposition}
\begin{proof}
Since $\hat{x}$ is a minimizer of $f$ on $X$, that means $\hat{x}\in S_f(y)$ for all $y\in X$, and consequently  $\langle y^*,\hat{x}-y\rangle\leq 0$ for all $y^*\in N(y)$. The result follows.
\end{proof}
It is important to notice that in the previous result, we do not require any assumptions of $f$ nor of $X$. On the other hand, the converse of the previous result is not true in general as we can see in the following example.
\begin{example}\label{Example1}
Consider $X=[-1,1]$ and $f:\R\to\R$ defined as 
\[
f(x)=\left\lbrace\begin{matrix}
1,&x\leq0\\
0,&x>0
\end{matrix}\right.
\]
It is not difficult to show that $N(x)=\{0\}$, for all $x\in \R$. Thus, the inequality \eqref{A1} holds for $x=-1$, but this point is not a minimizer of $f$.
\end{example}

We now establish the converse of Proposition \ref{A} in terms of $N^s$. In other words, we establish a sufficient condition to guarantee that a point is a minimizer of a real-valued function.
This result was inspired by Lemma 2.2 in \cite{Cavazzuti}.
\begin{proposition}\label{B}
Let $X$ be a convex and non-empty subset of $\R^n$ and $f:\R^n\to\R$ be a continuous and quasi-convex function. If $\hat{x}\in X$ satisfies
\begin{align}\label{B1}
\langle y^*,y-\hat{x}\rangle\geq0 \mbox{ for all }y\in X\mbox{ and all }y^*\in N^s(y);
\end{align}
then $x\in\argmin_X f$.
\end{proposition}
\begin{proof}
Suppose that $\hat{x}$ is not a minimizer of $f$ on $X$, that means there exists $y\in X$ such that $f(y)<f(\hat{x})$. By quasi-convexity and continuity there exists $t\in]0,1[$ such that
\[
f(y)<f(z)<f(\hat{x}),
\]
where $z=t\hat{x}+(1-t)y\in X$. Thus, $S_f^<(z)$ is convex and open, and this implies there exists $z^*\in N^s(z)$ such that $\langle z^*,y-z\rangle<0$, due to a separation theorem. Since $y-z=\dfrac{t}{1-t}(z-\hat{x})$ we obtain $\langle z^*,z-\hat{x}\rangle<0$, which is a contradiction. Hence, $\hat{x}$ is a minimizer of $f$ on $X$.
\end{proof}
The next example shows the converse implication of the previous result does not hold.
\begin{example}
Consider $X$ and $f$ given in Example \ref{Example1}. We can verify that 
\[
N^s(x)=\left\lbrace\begin{matrix}
[0,+\infty[,&x\leq 0\\
\R,&x>0
\end{matrix}\right.
\]
Furthermore, $x=1$ is a minimizer of $f$, but it does not verify inequality \eqref{B1}.
\end{example}

In the next result, we present necessary and sufficient conditions to guarantee a point to be a minimizer of a function using the adjusted normal cone.

\begin{proposition}\label{C}
Let $X$ be a non-empty subset of $\R^n$ and $f:\R^n\to\R$ be a function.
 If $\hat{x}\in\argmin_X f$, then
\begin{align}\label{A1}
\langle y^*,y-\hat{x}\rangle\geq0 \mbox{ for all }y\in X\mbox{ and all }y^*\in N^a(y).
\end{align}
The converse holds provided that $X$ is convex and $f$ is continuous and quasi-convex.
\end{proposition}
\begin{proof}
It follows from the same steps given in the proof of Propositions \ref{A} and \ref{B}.
\end{proof}

\begin{remark}
Propositions \ref{B} and \ref{C} are also true under pseudo-continuity, due to Proposition \ref{pseudo-CC}.  Indeed, if $f$ is pseudo-continuous then there exists an increasing function $h$ and a continuous function $g$ such that $f=h\circ g$. This implies that any level set of $f$ coincides with the level set of $g$. Consequently, they have the same set of minimizers and the same normal cones. Finally, we apply Propositions \ref{B} and \ref{C} to the function $g$, and the affirmation follows.
\end{remark}

We now recall continuity notions for set-valued maps. 

\,

Let $U, V$ be non-empty sets. A \emph{set-valued map} $T: U\tos V$ is an application $T:U\to \mathcal{P}(V)$, that is, for $u\in U$, $T(u)\subset V$. 
The \emph{graph} of $T$ is defined as
\[\gra(T):=\big\{(u,v)\in U\times V\::\: v\in T(u)\big\}.\]
Let $T: X\tos Y$ be a correspondence with $X$ and $Y$ two topological spaces.
The map $T$ is said  to be:
\begin{itemize}
 \item \emph{closed}, when $\gra(T)$ is a closed subset of $X\times Y$;
 \item \emph{lower semicontinuous} when for all $x_0\in X$ and any sequence $(x_n)_{n\in\N}$ converging to $x_0$ and any element $y_0$ of $T(x_0)$, there exists a sequence $(y_n)_{n\in\N}$ converging to $y_0$ such that $y_n\in T(x_n)$, for any $n\in\N$.
 \item \emph{upper semicontinuous} when for all $x\in X$ and any open set $V$, with $T(x)\subset V$,  there exists a neighbourhood $\mathscr{V}_x$ of $x$ such that $T(\mathscr{V}_x) \subset V$;
 \item \emph{continuous} when it is upper and lower semicontinuous.
 \end{itemize}
 
 We finish this section with the following result.
 
\begin{proposition}\label{convex-lower}
Let $T:[a,b]\tos \R^m$ be a set-valued map with non-empty values. If $\gra(T)$ is convex and $T(\{a,b\})$ is bounded, then it is lower semicontinuous.
\end{proposition}
\begin{proof}
Let $x\in [a,b]$, $(x_n)_{n\in\N}$ be a sequence converging to $x$ and $x^*\in T(x)$. For each $n\in\N$, there exist $t_n\in[0,1]$ and $c_n\in\{a,b\}\setminus\{x\}$ such that $x_n=t_nx+(1-t_n)c_n$. Since $T$ has non-empty values, we take $c_n^*\in T(c_n)$. By convexity of  $\gra(T)$, one has $x_n^*=t_nx^*+(1-t_n)c_n^*\in T(x_n)$. It is not difficult to show that the sequence $(t_n)$ converges to $ 1$. Thus, the sequence $(x_n^*)_{n\in\N}$ converges to $x^*$. Therefore, $T$ is lower semicontinuous. 
\end{proof} 


\section{The generalized Nash game proposed by Rosen}\label{motivation}
Let $N$ be the set of players which is any finite and non-empty set.
Let us assume that each player $\nu \in N$ chooses a strategy $x^\nu $ in a strategy set $K_{\nu}$, which is a subset of  $\R^{n_\nu}$. We denote by $\R^n$, $K$ and $K_{-\nu}$ the Cartesian products of $\prod_{\nu\in N}\R^{n_\nu}$,  $\prod_{\nu\in N} K_{\nu}$ and $\prod_{\mu\in N\setminus\{\nu\}} K_{\mu}$, respectively. We can write $x = (x^\nu, x^{-\nu}) \in K$ in order to emphasize the strategy of player $\nu$, $x^\nu\in K_\nu$, and the strategy of the other players $x^{-\nu}\in K_{-\nu}$.

\,

Given the strategy of all players except for player $\nu$, $x^{-\nu}$,  player $\nu$ chooses a strategy $x^\nu$ such that it solves the following optimization problem  
\begin{align}\label{NEP}
\min_{x^\nu} \theta_\nu(x^\nu,x^{-\nu}),~\mbox{ subject to }~x^\nu\in K_\nu,
\end{align}
where $\theta_\nu: \R^n\to\R$ is a real-valued function and  $\theta_\nu(x^\nu,x^{-\nu})$ denotes the loss player $\nu$ suffers when the rival players have chosen 
the strategy $x^{-\nu}$.  Thus, a  {\em Nash equilibrium}  is a vector $\hat{x}$ such that $\hat{x}^\nu$ solves \eqref{NEP} when the rival players take the strategy $\hat{x}^{-\nu}$, for any $\nu$. 
We denote by $\nep(\{\theta_\nu,K_\nu\}_{\nu\in N})$ the set of Nash equilibria. 

\,

The following is a classic result of the existence of Nash equilibria.

\begin{theorem}[Debreu, Glicksberg, Fan]\label{D}
Suppose for each $\nu\in N$, $K_\nu$ is a compact, convex and non-empty set, the objective function $\theta_\nu$ is continuous and quasi-convex concerning its player's variable. Then, the set $\nep(\{\theta_\nu,K_\nu\}_{\nu\in N})$ is non-empty.
\end{theorem}

In a generalized Nash equilibrium problem, the strategy of each player must belong to a set $X_\nu(x^{-\nu})\subset K_\nu$ that depends on the rival players' strategies. The aim of player $\nu$, given the others players' strategies $x^{-\nu}$, is to choose a strategy $x^\nu$ that solves the next minimization problem
\begin{align}\label{GNEP}
\min_{x^\nu} \theta_\nu(x^\nu,x^{-\nu}),~\mbox{ subject to }~x^\nu\in X_\nu(x^{-\nu}),
\end{align}
where $X_\nu$ is a set-valued map from $K_{-\nu}$ to $K_\nu$. Thus, a vector $\hat{x}$ is a \emph{generalized Nash equilibrium} if, $\hat{x}^\nu$ solves \eqref{GNEP} when the rival players take the strategy $\hat{x}^{-\nu}$, for any $\nu$. 
We denote by $\gnep(\{\theta_\nu,X_\nu\}_{\nu\in N})$ the set of
generalized Nash equilibria. 

\,

The following result is about the existence of generalized Nash equilibria due to Arrow and Debreu \cite{Arrow-Debreu}, but we state it as in \cite{Facchinei2007}.

\begin{theorem}[Arrow and Debreu]\label{A-D}
Suppose for each $\nu\in N$, $K_\nu$ is a compact, convex and non-empty set, the objective function $\theta_\nu$ is continuous and quasi-convex concerning its player's variable, and the constraint map $X_\nu$ is continuous with convex, compact and non-empty values. Then, the set $\gnep(\{\theta_\nu,X_\nu\}_{\nu\in N})$ is non-empty.
\end{theorem}

\begin{remark}
The previous results are also true by considering pseudo-continuity instead of continuity, this was proved by Morgan and Scalzo \cite{MORGAN2007}. However, thanks to Proposition \ref{pseudo-CC}, they are equivalent in the sense that we can prove one of them from the other, and this was established in \cite{Co21}.
\end{remark}

An important instance of a generalized Nash equilibrium problem was presented by Rosen in \cite{Rosen}. More specifically, let $X$  be a convex and non-empty subset of $\R^n$. For each $\nu\in N$ we define
\begin{align}\label{R-X}
X_\nu(x^{-\nu}):=\{x^\nu\in\R^{n_\nu}:~(x^\nu,x^{-\nu})\in X\}.
\end{align}
Here $K_\nu$ is the projection of $X$ onto $\R^{n_\nu}$, that is
\begin{align}\label{Projection-X}
K_\nu=\{x^\nu\in\R^{-\nu}:~(x^\nu,x^{-\nu})\in X\mbox{ for some }x^{-\nu}\in\R^{n-n_\nu}\}
\end{align}
It is not difficult to see that in general the sets $X$ and $K$ are different. 
In order to illustrate it we can see Figure \ref{F1}.
\begin{figure}[h!]
\centering
\begin{tikzpicture}[>=latex,scale=1.35]
\draw[->](-0.25,0)--(4,0)node[below right]{$x^1$};
\draw[->](0,-0.25)--(0,3)node[left]{$x^2$};
\draw[fill=gray!30,line width=0.7](1,1.5)--(2,2.5)--(3,1.5)--(2,0.5)--(1,1.5);
\draw(2.35,2.5)node[]{$X$};
\draw[dotted](1,1.5)--(1,0);
\draw[dotted](3,1.5)--(3,0);
\draw[line width=0.7](1,-0.03)--(3,-0.03)node[below right]{$K_1$};
\draw[dotted](0,0.5)--(2,0.5);
\draw[dotted](0,2.5)--(2,2.5);
\draw[line width=0.7](-0.03,0.5)--(-0.03,2.5)node[left]{$K_2$};
\draw[line width=0.7](2.5,1)--(2.5,2);
\draw[dotted](2.5,0)--(2.5,1);
\draw(2.5,0)node[below]{$x^1$};
\draw[<-](2.52,1.6)--(3,2.25)node[right]{$X_2(x^1)$};
\draw[line width=0.7](1.25,1.25)--(2.75,1.25);
\draw[dotted](0,1.25)--(1.25,1.25);
\draw(0,1.25)node[left]{$x^2$};
\draw[<-](2,1.2)--(1,0.7)node[left]{$X_1(x^2)$};
\end{tikzpicture}
\caption{$X\subset\R^2, K_1,K_2, X_1(x^2)$ and $X_2(x^1)$}\label{F1}
\end{figure}
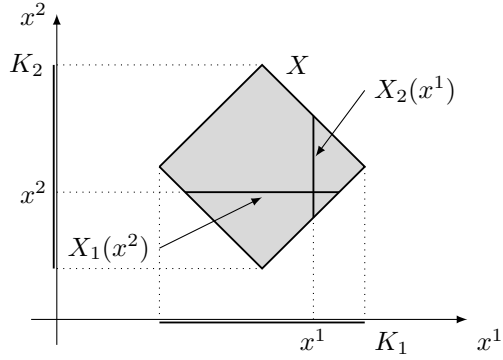

We denote by $\Rgnep(\{\theta_\nu,X\}_{\nu\in N})$ the solution set of this generalized Nash equilibrium problem proposed by Rosen and we state below Rosen's theorem.

\begin{theorem}[Rosen]\label{TR}
Assume that $X$ is a convex, compact and non-empty subset of $\R^n$. 
If for each $\nu\in N$ the objective function $\theta_\nu$ is continuous and convex concerning its player's variable, then the set $\Rgnep(\{\theta_\nu,X\}_{\nu\in N})$ is non-empty.
\end{theorem}

Since any convex function is quasi-convex, a natural question arises: is Theorem \ref{TR} a consequence of  Theorem \ref{A-D}?
This question is motivated by the following proposition.
\begin{proposition}\label{mot-Rosen}
Let $X$ be a convex, compact and non-empty subset  of  $\R^2$ and
 $\theta_1,\theta_2$ be two functions defined from $\R^2$ onto $\R$. If each function $\theta_\nu$ is continuous and quasi-convex concerning its player's variable, then the set $\Rgnep(\{\theta_\nu,X\}_{\nu\in \{1,2\}})$ is non-empty.
\end{proposition}
\begin{proof}
It is clear that, for each $\nu\in \{1,2\}$,  the set $K_\nu$  defined as in \eqref{Projection-X} is just a compact interval of $\R$. On the other hand, the map $X_\nu:K_{-\nu}\tos K_\nu$ defined as in \eqref{R-X} is upper semicontinuous with compact, convex and non-empty values. Moreover, it is also lower semicontinuous, due to Proposition \ref{convex-lower}. Thus, the existence of generalized Nash equilibria is guaranteed by Theorem \ref{A-D}.
\end{proof}

Additionally, Ponstein \cite{Ponstein} gives a positive answer to this question when $X$ is a convex polyhedron (see Lemma 2 in \cite{Ponstein}). However, the following example says that Theorem \ref{TR} is not a direct consequence of Theorem \ref{A-D}.

\begin{example}\label{example1}
Consider the sets
\[
A:=[0,e_3]\cup\{(x,y,0)\in\R^3:~x^2+(y-1)^2\leq 1~\wedge~x\geq0\}
\]
and $X:=\co(A)$, the convex hull of $A$, see Figure \ref{set-X}.
\begin{figure}[h!]
\centering
\begin{tikzpicture}[>=latex,scale=1.5]

\draw[<-,gray](-0.75,-0.75)--(0,0)node at (-0.75,-0.75)[below left,black]{$x$};
\draw[->,gray](0,0)--(2.5,0)node[below right,black]{$y$};
\draw[->,gray](0,0)--(0,2)node[left,black]{$z$};
\draw[line width=0.7](0,0)--(0,1);
\draw[line width=0.6](0,0)--(2,0);
\draw[line width=0.7,domain=0:2]plot(\x,{0.35*(\x-1)^2-0.35});
\end{tikzpicture}
\hfill
\begin{tikzpicture}[>=latex,scale=1.5]
\fill[gray!60,domain=0:2](0,1)--plot(\x,{0.35*(\x-1)^2-0.35})--(0,1);
\draw[<-,gray](-0.75,-0.75)--(0,0)node at (-0.75,-0.75)[below left,black]{$x$};
\draw[->,gray](0,0)--(2.5,0)node[below right,black]{$y$};
\draw[->,gray](0,0)--(0,2)node[left,black]{$z$};
\draw[line width=0.7](0,0)--(0,1);
\draw[line width=0.6,dashed](0,0)--(2,0);
\draw[line width=0.7,domain=0:2]plot(\x,{0.35*(\x-1)^2-0.35});
\draw[line width=0.7](0,1)--(2,0);
\draw[line width=0.7](0,1)--(1,-0.35);
\end{tikzpicture}
\caption{The set $A$ and its convex hull $X$}\label{set-X}
\end{figure}
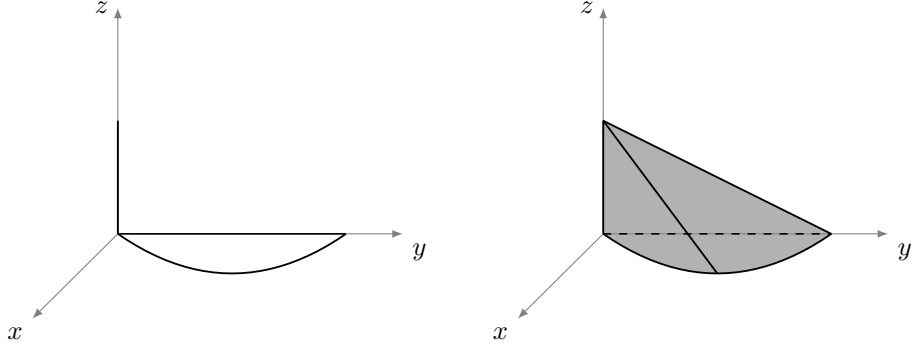

Moreover, consider two continuous functions $\theta_1,\theta_2:\R^3\to\R$ which are convex concerning their player's variable. 

Here, we note that $K_1=\{(x,y)\in\R^2:~x^2+(y-1)^2\leq1~\wedge~x\geq0\}$ and $K_2=[0,1]$. Thus, for each $(x,y)\in K_1$ and $z\in K_2$ we have
\[
X_1(z)=\{(x,y)\in\R^2:~x^2+(y-(1-z))^2\leq (1-z)^2~\wedge~x\geq0\}
\]
and
\[
X_2(x,y)=\left\lbrace\begin{matrix}
[0,1],&(x,y)=(0,0)\\
\left[0,1-\dfrac{x^2+y^2}{2y}\right],&(x,y)\neq(0,0)
\end{matrix}\right.
\]
By Proposition \ref{convex-lower}, the map $X_1$ is lower semicontinuous. However, 
it was showed in \cite{Piotr} that $X_2$ is not lower semicontinuous.
Thus, we cannot directly apply  Theorem \ref{A-D} to guarantee the existence of solutions for the RGNEP. 
\end{example}

It is important to note that  the proof of Theorems \ref{D}, \ref{A-D} and \ref{TR} consists of reformulating the games as a fixed point problem in order to apply the famous  Kakutani's theorem. 
In Figure \ref{Implications 1}, we present the links between these results. 

\begin{figure}[h!]
\centering
\begin{tikzpicture}[>=latex,scale=1]
\draw(-4.5,1)node[]{Kakutani's theorem};
\draw[->](-3,1)--(-2,1);
\draw(-0.25,1)node[left]{Theorem \ref{A-D}};
\draw[->](0,1)--(1,1);
\draw(2,1)node[]{Theorem \ref{D}};
\draw(-4.5,-0.3)node[]{Theorem \ref{TR}};
\draw[->](-4.4,0.7)--(-4.4,-0.1);
\end{tikzpicture}
\caption{Kakutani's  theorem and existence results of (generalized) Nash equilibria.}\label{Implications 1}
\end{figure}
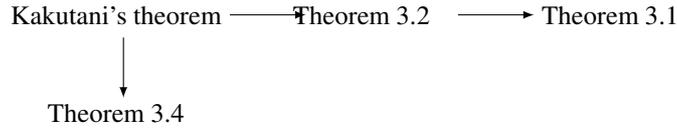

On the other hand, Yu \emph{et al.} \cite{Yu-Fa-Yang2016} showed that the fixed point theorem due to Kakutani is a consequence of Theorem \ref{D}. The authors in \cite{Yu-Fa-Yang2016} reformulated the fixed point problem as a classical Nash game.  Thus, following the chain of implications we deduce that Theorem \ref{TR} follows from Theorem \ref{D}. Hence, the answer to our question is positive. However, 
 we will give a direct proof of Theorem \ref{TR} as a consequence of Theorem \ref{D}. 
First consider the following  functions $f_1,f_2:\R^n\times \R^n\to\R$ defined as
\begin{align}\label{neo-games}
f_1(x,y):=\|x-y\|\mbox{ and }f_2(x,y):=\sum_{\nu\in N}\theta_\nu(y^\nu,x^{-\nu}),
\end{align}
where $\|\cdot\|$ is a norm in $\R^n$.

\,

The following result establishes that in order to find a solution of the RGNEP, it is enough to solve a particular Nash equilibrium problem.

\begin{proposition}\label{P1}
Assume that $X$ is a non-empty subset of $\R^n$ and
$K_1=K_2=X$. If $(\hat{x},\hat{y})\in \nep(\{f_i,K_i\}_{i\in\{1,2\}})$ then $\hat{x}\in\Rgnep(\{\theta_\nu,X\}_{\nu\in N})$.
\end{proposition}
\begin{proof}
Clearly, $(\hat{x},\hat{y})\in \nep(\{f_i,K_i\}_{i=1,2})$ if, and only if, 
\[
\|\hat{x}-\hat{y}\|\leq\|x-\hat{y}\|\mbox{ and }f_2(\hat{x},\hat{y})\leq f_2(\hat{x},y)\mbox{ for all }x,y\in X.
\]
This is equivalent to the following
\[
\hat{x}=\hat{y}\mbox{ and }\sum_{\nu\in N}\theta_\nu(\hat{y}^\nu,\hat{x}^{-\nu})\leq \sum_{i\in N}\theta_\nu(y^\nu,\hat{x}^{-\nu}),\mbox{ for all }y\in X.
\]
Now, for each $\nu\in N$ we take $y=(y^\nu,\hat{x}_{-\nu})$ and replace it in the previous inequality, thus
\[
\theta_\nu(\hat{x})\leq \theta_\nu(y^\nu,\hat{x}^{-\nu}), \mbox{ for all }y^\nu\mbox{ such that }(y^\nu,\hat{x}^{-\nu})\in X.
\]
This completes the proof.
\end{proof}
We are ready to give a positive answer to our question.

\begin{theorem}
Theorem \ref{D} implies Theorem \ref{TR}.
\end{theorem}
\begin{proof}
The result follows from Proposition \ref{P1} and Theorem \ref{D}.
\end{proof}

\begin{remark}
A few remarks are needed.
\begin{enumerate}

\item The function $f_2$ in \eqref{neo-games} was used by Rosen \cite{Rosen}  to reformulate the RGNEP as a fixed point problem. 

\item Since quasi-convexity is not preserved by the sum, a natural question arises: Is it possible to reduce the RGNEP to a classical NEP, under quasi-convexity? The answer to this question is given by the following chain of implications
\begin{figure}[h!]
\centering
\begin{tikzpicture}[>=latex]
\draw(0.1,1)node[left]{Theorem \ref{D}};
\draw[->](0.2,1)--(1,1);
\draw(1.25,1)node[right]{Theorem \ref{VIP}};
\draw[->](2.8,1)--(3.6,1);
\draw(3.75,1)node[right]{Theorem \ref{main-result}};
\end{tikzpicture}
\end{figure}
where the second implication was proved by Bueno \emph{et al.} in \cite{BCC21} and they reduced the RGNEP to a variational inequality. Later, the first implication was given by
Yu \emph{et al.} in \cite{Yu-Fa-Yang2016} and they reformulated the Stampacchia variational inequality as a classical Nash game with three players. 

\item Finally, the converse of Proposition \ref{P1} is not true in general. Indeed, consider Example 1 in \cite{AAD2014}, that is the RGNEP defined by 
\[
X=\{(x^1,x^2)\in\R^2:~x^1\geq0,~x^2\geq0\mbox{ and }2x^1+x^2\leq1\}
\]
and the functions $\theta_1.\theta_2:\R^2\to\R$ defined as
\[
\theta_1(x^1,x^2)=(x^1-2)^2\mbox{ and }\theta_2(x^1,x^2)=(x^2-2)^2. 
\]
Then
\[
\Rgnep(\{\theta_\nu,X\}_{\nu\in \{1,2\}})=\{(x^1,x^2)\in\R^2:~x^1\geq0,~x^2\geq0\mbox{ and }2x^1+x^2=1\}.
\]
Moreover, by considering the Euclidean norm in $\R^2$, the functions $f_1$ and $f_2$ defined in \eqref{neo-games} are given by
\[
f_1((x^1,x^2),(y^1,y^2))=\sqrt{(x^1-y^1)^2+(x^2-y^2)^2}\mbox{ and }
\]
\[
f_2((x^1,x^2),(y^1,y^2))=(y^1-2)^2+(y^2-2)^2.
\]
Thus, $\nep(\{f_i,K_i\}_{i\in \{1,2\}})=\{((0,1),(0,1))\}$. 

On the other hand, Figure \ref{contraex} shows the sets $X$ and
$\Rgnep(\{\theta_\nu,X\}_{\nu\in \{1,2\}})$, and also the minimizer of the function $f_2((x^1,x^2),\cdot)$.
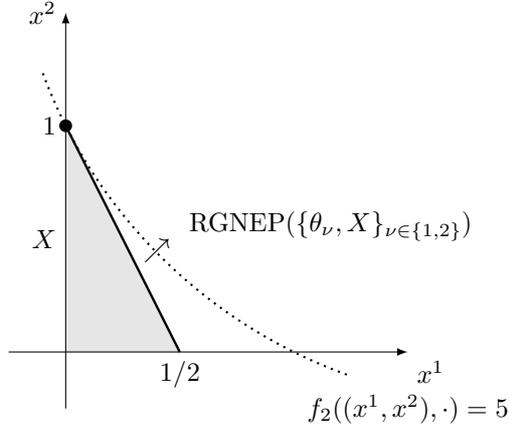
\begin{figure}[h!]
\centering
\begin{tikzpicture}[>=latex,scale=3]
\fill[gray!20](0,0)--(0,1)--(0.5,0)--(0,0);
\draw[line width=0.9](0,1)--(0.5,0);
\draw(0.3,0.45)node[right]{$\nearrow$};
\draw(0,0.5)node[left]{$X$};
\draw[->](-0.25,0)--(1.5,0)node[below right]{$x^1$};
\draw[->](0,-0.25)--(0,1.5)node[left]{$x^2$};
\draw(0.5,0.45)node[above right]{$\Rgnep(\{\theta_\nu,X\}_{\nu\in \{1,2\}})$};
\fill(0,1)circle(0.8pt)node[left]{$1$};
\draw(0.5,0)node[below]{$1/2$};
\draw[domain=-0.1:1.25,samples=100,dotted,line width=0.8]plot(\x,{2-sqrt(5-(\x-2)^2)});
\draw(1.5,-0.15)node[below]{$f_2((x^1,x^2),\cdot)=5$};
\end{tikzpicture}
\caption{The sets $X$ and $\Rgnep(\{\theta_\nu,X\}_{\nu\in \{1,2\}})$} \label{contraex}
\end{figure}
\end{enumerate}
\end{remark}

\section{Variational inequality approach}\label{variational approach}

The concept of variational inequality has proven to be an important tool in optimization, complementary problems, game theory among others. The purpose of this section is to present and existence result under discontinuity assumption and to characterize the solutions of  RGENPs by means of solutions of certain Minty variational inequalities.

\,

Given a set-valued map $T:\R^m\tos\R^m$ and a set $X\subset \R^m$, a point $x\in X$ is said to be a solution of the \emph{Stampacchia variational inequality problem} $SVI(T,X)$ if, there exists $x^*\in T(x)$ such that
\[
\langle x^*,y-x\rangle\geq0,\mbox{ for all }y\in X.
\]
Similarly, $x\in X$ is said to be a solution of the \emph{Minty variational inequality problem} $MVI(T,X)$ if,
\[
\langle y^*,y-x\rangle\geq 0,\mbox{ for all }y\in X\mbox{ and all }y^*\in T(y).
\]
Here, we use notations $SVI(T,X)$ and $MVI(T,X)$  for the problems themself and their solution sets. 

\,

A classic result of the existence of solutions for variational inequalities is Theorem 9.9 in 
\cite{JP-Aubin-1998} and we state it as follows.

\begin{theorem}\label{VIP}
Assume that $X$ is a compact, non-empty and convex subset of $\R^m$ and $T:\R^m\tos\R^m$ is a set-valued map. If $T$ is upper semicontinuous with compact, non-empty and convex values, then the $SVI(T,X)$ admits at least one solution.
\end{theorem}

Some existence results for the Minty variational inequality problem can be found in \cite{AC-2013,AH04}. Important characterizations of generalized monotonicity were related to the non-emptiness of the solution set for Minty variational inequality problems, see \cite{John2001}.

\,

We now introduce some elements that we will need in the next subsections. Associated to the RGNEP defined by $X\subset\R^n$ and $\{\theta_{\nu}\}_{\nu\in N}$, for each $\nu\in N$ and each $x\in \R^n$, we consider the sets
\[
L_\nu(x):=S_{\theta_\nu(\cdot,x^{-\nu})}(x^\nu)
,~
L_\nu^s(x):=S^<_{\theta_\nu(\cdot,x^{-\nu})}(x^\nu)\mbox{ and }
L_\nu^a(x):=S^a_{\theta_\nu(\cdot,x^{-\nu})}(x^\nu).
\]
Additionally, we consider their respective normal cone $N_\nu(x),~N_\nu^< (x)$ and $N_\nu^a(x)$.
We finish by considering the set-valued maps $\mathcal{N},\mathcal{N}^s,\mathcal{N}^a:\R^n\tos\R^n$ defined as
\begin{align}\label{Maps-N}
\mathcal{N}(x):=\prod_{\nu\in N}N_\nu(x),~\mathcal{N}^s(x):=\prod_{\nu\in N}N_\nu^s(x)
\mbox{ and }\mathcal{N}^a(x):=\prod_{\nu\in N}N_\nu^a(x).
\end{align}

\subsection{Stampacchia variational inequality}\label{Stampacchia}
Inspired by Harker \cite{HARKER199181}, some authors prove the existence of generalized Nash equilibria for the RGNEP by solving certain Stampacchia variational inequality problems. In that sense, Facchinei \emph{et al.} \cite{FACCHINEI2007159} gave the following result.

\begin{theorem}[Facchinei \emph{et al.}]\label{FFP}
Assume that $X$ is a convex, compact and non-empty subset of $\R^n$. If each objective function is continuously differentiable and pseudo-convex concerning its player's variable, then the set $\Rgnep(\{\theta_\nu, X\}_{\nu\in N})$ is non-empty.
\end{theorem}

In  the previous result, Facchinei \emph{et al.} \cite{FACCHINEI2007159} reduced the RGNEP to a Stampacchia variational inequality problem with $X$ and the map $T:\R^n\to\R^n$ defined as
\begin{align}\label{T-Facchinei}
T(x):=\left( \begin{matrix}
\nabla_{x_1}\theta_1(x)\\
\vdots\\
\nabla_{x_p}\theta_p(x)
\end{matrix}\right).
\end{align}
In this case, since each function is continuous differentiable the map $T$ is continuous. Thus, the set $S(T,X)$ is non-empty. Finally, pseudo-convexity implies that any element of $S(T,X)$ is a generalized Nash equilibrium.

\,

Later, Aussel and Dutta \cite{Aussel-Dutta} established the following result, which generalizes Theorems \ref{TR} and \ref{FFP}. 

\begin{theorem}[Aussel and Dutta]\label{AD08}
Assume that $X$ is a convex, compact and non-empty subset of $\R^n$. If each objective function is continuous and semi-strictly quasi-convex concerning its player's variable, then the set $\Rgnep(\{\theta_\nu, X\}_{\nu\in N})$ is non-empty.
\end{theorem}

The authors in \cite{Aussel-Dutta} also reduced the RGNEP to a Stampacchia variational inequality problem with $X$ and the set-valued map $T_0:\R^n\tos\R^n$ defined as
\[
T_0(x):=\prod_{\nu\in N}F_\nu(x),
\]
where 
\[
F_\nu(x):=\left\lbrace\begin{matrix}
\overline{B_\nu}(0,1),&x_\nu\in A_\nu(x_{-\nu})\\
\co(N^a_{\nu}(x)\cap S_\nu(0,1)),&\mbox{otherwise}
\end{matrix}\right.
\]
with $\overline{B}_\nu(0,1)=\{z\in\R^{n_\nu}:~\|z\|\leq1\}$, $
S_\nu(0,1)=\{z\in\R^{n_\nu}:~\|z\|=1\}$  and $A_\nu(x_{-\nu})=\argmin_{\R^{n_\nu}}\theta_\nu(\cdot,x_{-\nu})$.

\,

The idea of the proof in the previous result basically consists on showing that the map $T_0$ is upper semicontinuous to guarantee the existence of solution of the variational inequality associated. By using the semi strict quasi-convexity, any solution of this variational inequality is a generalized Nash equilibrium.

\,

Recently, inspired by the previous result, Bueno \emph{et al.}  \cite{BCC21} showed an existence result for this generalized Nash game proposed by Rosen in the general setting of  quasi-convexity. We state it as follows.
\begin{theorem}[Bueno \emph{et al.}]\label{main-result}
Assume that $X$ is a convex, compact and non-empty subset of $\R^n$. If each objective function is continuous and quasi-convex concerning its player's variable, then the set $\Rgnep(\{\theta_\nu, X\}_{\nu\in N})$ is non-empty.
\end{theorem}

The authors in \cite{BCC21} also reduced the RGNEP as a variational inequality problem. In this case, they consider  the set-valued map $T_1:\R^n\tos\R^n$ defined as 
\begin{align}\label{T1}
T_1(x):=\prod_{\nu\in N}G_\nu(x)
\end{align}
where $G_\nu(x)=\co(N^<_\nu(x)\cap S_\nu(0,1))$.
Furthermore, it is not difficult to see that $\gra(T_0)\subset \gra(T_1)$ and this inclusion can be strict. Indeed, consider for instance the RGNEP defined by $X\subset\R^3$
and the functions $\theta_1,\theta_2:\R^3\to\R$ defined by
\[
\theta_1((x^1,y^1),x^2)=\left\lbrace\begin{matrix}
|x^1|+|y^1|,&|x^1|+|y^1|\leq1\\
1,&\mbox{otherwise}
\end{matrix}
\right.
\mbox{ and }\theta_2((x^1,y^1),x^2)=x^2
\]
Both functions are continuous and quasi-convex, but $\theta_1$ is not semi strictly quasi-convex. Moreover, it is not difficult to show that 
\[
(1,2)\in N_1^<((10,0),0)\setminus N_1^a((10,0),0).
\]
Consequently,  $\gra(T_0)\subset \gra(T_1)$ and $T_0\neq T_1$.

\begin{remark}
The existence of generalized Nash equilibria in Example \ref{example1} is guaranteed by Theorem \ref{main-result}. Furthermore, this generalizes Theorems \ref{AD08}, \ref{FFP} and Theorem \ref{D}. On the other hand, since Theorem \ref{D} does not require semi strict quasi-convexity, this is not a consequence of Theorem \ref{AD08}. 
 In Figure \ref{Implications 2}, we represent the links between these results.
\begin{figure}[h!]
\centering
\begin{tikzpicture}[>=latex]
\draw(0,1)node[left]{Theorem \ref{VIP}};
\draw[->](1.75,0.8)--(1.75,0.2);
\draw[->](0.1,1)--(0.9,1);
\draw(1,1)node[right]{Theorem \ref{main-result}};
\draw(1,0)node[right]{Theorem \ref{D}};
\draw[->](2.5,1)--(3.3,1);
\draw(3.4,1)node[right]{Theorem \ref{AD08}};
\draw[->](4.9,1)--(5.7,1);
\draw(5.7,1)node[right]{Theorem \ref{FFP}};
\end{tikzpicture}
\caption{Relationship between Theorems \ref{D}, \ref{main-result},  \ref{AD08} and \ref{FFP} }\label{Implications 2}
\end{figure}
\end{remark}

In the following result, we consider a generalized Nash game proposed by Rosen with discontinuous functions, which is a consequence of Theorem \ref{main-result} and Proposition \ref{pseudo-CC}. 

\begin{theorem}\label{Rosen-pseudocontinuous}
Assume that $X$ is a convex, compact and non-empty subset of $\R^n$. If each objective function is pseudo-continuous and quasi-convex concerning its player's variable, then the set $\Rgnep(\{\theta_\nu, X\}_{\nu\in N})$ is non-empty.
\end{theorem}
\begin{proof}
Thanks to Proposition \ref{pseudo-CC}, for each player $\nu$ there exists a continuous function $\varphi_\nu:\R^n\to\R$ and an increasing function $h_\nu: \varphi_\nu(\R^n)\to\R$ such that
\[
\theta_\nu=h_\nu\circ \varphi_\nu.
\]
It is not difficult to show that $\varphi_\nu$ is quasi-convex concerning its variable's player and also $\Rgnep(\{\theta_\nu,X\}_{\nu\in N})=\Rgnep(\{\varphi_\nu,X\}_{\nu\in N})$. Therefore, the result follows from Theorem \ref{main-result}.
\end{proof}
As a direct consequence of the above result, we recover Theorem 3.2 in \cite{MORGAN2007}, which is stated below.
\begin{corollary}
Suppose for each $\nu\in N$, $K_\nu$ is a compact, convex and non-empty set, the objective function $\theta_\nu$ is pseudo-continuous and quasi-convex concerning its player's variable. Then, the set $\nep(\{\theta_\nu,K_\nu\}_{\nu\in N})$ is non-empty.
\end{corollary}

\subsection{Minty variational inequality}
Cavazzuti \emph{et al.} \cite{Cavazzuti} considered for each $x\in X\subset\R^n$ the set
\begin{align}\label{map-S}
\mathcal{S}(x):=\bigcup_{\nu\in N}X_\nu(x^{-\nu})\times\{x^{-\nu}\}.
\end{align}
It is not difficult to verify that $\mathcal{S}(x)\subset \mathcal{X}(x)$, for all $x\in X$, where  $\mathcal{X}$ is the map defined as
\[
\mathcal{X}(x):=\prod_{\nu\in N} X_\nu(x^{-\nu}).
\] 
Figure \ref{Cava} gives us an example about it on  $\R^2$.
\begin{figure}[h!]
\centering
\begin{tikzpicture}[scale=1.6,>=latex]
\fill[gray!20](-1,0)--(0,1)--(1,0)--(0,-1)--(-1,0);
\draw(-1,0)--(0,1)--(1,0)--(0,-1)--(-1,0);
\draw(-0.25,0.95)node[left]{$X$};
\fill(1,0)circle(1pt)node[below]{$1$};
\fill(-1,0)circle(1pt)node[below]{$-1$};
\fill(0,1)circle(1pt)node[left]{$1$};
\fill(0,-1)circle(1pt)node[left]{$-1$};
\draw[line width=0.7](-0.5,0.5)--(0.5,0.5)--(0.5,-0.5);
\fill(0.5,0.5)circle(1pt)node[right]{$x$};
\draw[->,gray](-1.5,0)--(1.5,0)node[below right,black]{$x^1$};
\draw[->,gray](0,-1.5)--(0,1.5)node[left,black]{$x^2$};
\draw(0.6,0.6)node[above right]{$ \mathcal{S}(x)$};
\end{tikzpicture}
\hfill
\begin{tikzpicture}[scale=1.6,>=latex]
\fill[gray!20](-1,0)--(0,1)--(1,0)--(0,-1)--(-1,0);
\draw(-1,0)--(0,1)--(1,0)--(0,-1)--(-1,0);
\draw(-0.25,0.95)node[left]{$X$};
\fill(1,0)circle(1pt)node[below]{$1$};
\fill(-1,0)circle(1pt)node[below]{$-1$};
\fill(0,1)circle(1pt)node[left]{$1$};
\fill(0,-1)circle(1pt)node[left]{$-1$};
\fill[gray!50](-0.5,0.5)--(0.5,0.5)--(0.5,-0.5)--(-0.5,-0.5)--(-0.5,0.5);
\draw[line width=0.7](-0.5,0.5)--(0.5,0.5)--(0.5,-0.5)--(-0.5,-0.5)--(-0.5,0.5);
\fill(0.5,0.5)circle(1pt)node[right]{$x$};
\draw[->,gray](-1.5,0)--(1.5,0)node[below right,black]{$x^1$};
\draw[->,gray](0,-1.5)--(0,1.5)node[left,black]{$x^2$};
\draw(0.6,0.6)node[above right]{$ \mathcal{X}(x)$};
\end{tikzpicture}
\caption{Sets  $ \mathcal{S}(x)$ and  $ \mathcal{X}(x)$.}\label{Cava}
\end{figure}
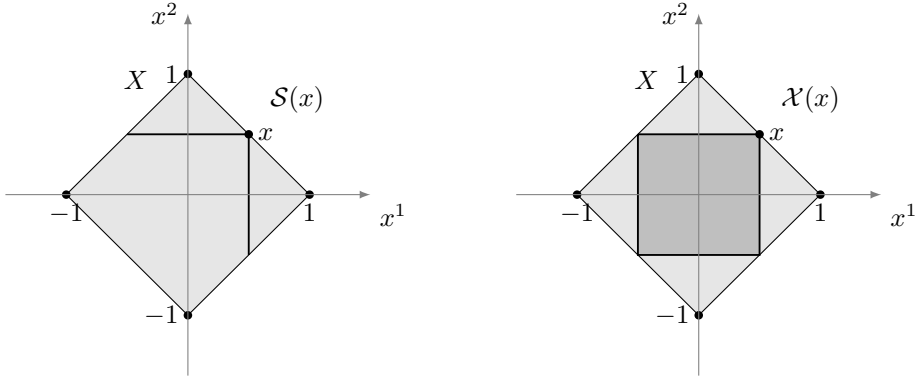

The following result is Theorem 3.2 in \cite{Cavazzuti}, which establishes a sufficient and necessary  condition for a point to be a generalized Nash equilibrium.
\begin{theorem}[Cavazzuti \emph{et al.} ]\label{Teo-Cava}
Assume that $X$ is a convex and non-empty subset of $\R^n$ and each objective function is differentiable. If $\hat{x}$ is a solution of $MVI(T,\mathcal{S}(\hat{x}))$, that is
\begin{align}\label{Mspecial}
\langle T(y),y-\hat{x}\rangle\geq0 \mbox{ for all }y\in \mathcal{S}(\hat{x}),
\end{align}
where $T$ is defined as in \eqref{T-Facchinei}.  
The converse holds if, each $\theta_\nu $ is quasi-convex in its player's variable.
\end{theorem}
The inequality given in \eqref{Mspecial} is known in the literature as \emph{Minty quasi-variational inequality}.

\,

As a direct consequence of the above result, we have the following corollary.
\begin{corollary}\label{coroCava}
Assume that $X$ is a convex and non-empty subset of $\R^n$ and each objective function is differentiable. If $\hat{x}\in X$  is a solution of the  $MVI(T,X)$; then  $\hat{x}\in\Rgnep(\{\theta_\nu, X\}_{\nu\in N})$.
\end{corollary}
The converse of Corollary \ref{coroCava} is not true in general as  the following example shows.
\begin{example}
Consider the two-game with $X=[-1,1]^2\subset\R^2$ and the functions $\theta_1,\theta_2:\R^2\to\R$ defined as
\[
\theta_1(x^1,x^2)=(x^1x^2)^2~\mbox{ and }~\theta_2(x^1,x^2)=(x^2)^3.
\]
Both functions are continuously differentiable and quasi-convex concerning their player's variable. Moreover, it is not difficult to show that $(0,-1)$ is a Nash equilibrium. 

On the other hand, the map $T$, defined as in \eqref{T-Facchinei}, is given by
\[
T(x^1,x^2)= (2x^1(x^2)^2,3(x^2)^2)
\]
and for $(x^1,x^2)=(0,1/2)$ we have  $\langle (0,3/4),(0,1/2)-(0,-1)\rangle=-3/8$,
which shows that $(0,-1)$ is not a solution of the $MVI(T,X)$.
\end{example}

We finish this section with the following result, which is an extension of Theorem \ref{Teo-Cava}. 
\begin{proposition}
Assume that $X$ is a convex and non-empty subset of $\R^n$, $\mathcal{S}$ is defined as in \eqref{map-S}, and $\mathcal{N}$, $\mathcal{N}^s$ and $\mathcal{N}^a$ are defined as in \eqref{Maps-N}, the following implications hold
\begin{enumerate}
\item If $\hat{x}\in \Rgnep(\{\theta_\nu, X\}_{\nu\in N})$  then 
$\hat{x}\in MVI(\mathcal{N},\mathcal{S}(\hat{x}))$.
\item Assume that each objective function is continuous and quasi-convex concerning its player's variable. If $\hat{x}\in MVI(\mathcal{N}^s,\mathcal{S}(\hat{x}))$,
then $\hat{x}\in \Rgnep(\{\theta_\nu, X\}_{\nu\in N})$.
\item If $\hat{x}\in \Rgnep(\{\theta_\nu, X\}_{\nu\in N})$  then $\hat{x}\in MVI(\mathcal{N}^a, \mathcal{S}(\hat{x}))$.
The converse holds provided that each objective function is continuous and quasi-convex concerning its player's variable. 
\end{enumerate}
\end{proposition}
\begin{proof}
Part {\it 1.} follows from Proposition \ref{A} and part {\it 2.} is a consequence of Proposition \ref{B}. Finally, part {\it 3.} follows from Proposition \ref{C}.
\end{proof}

\section{Coerciveness conditions}\label{coerciveness}

In a similar way to \cite{AS-2017,CHS-2020,CZ-2018},
this section is devoted to the study of solutions for the generalized Nash equilibrium
problem proposed by Rosen on unbounded strategy sets. We propose three different coerciveness conditions and we compare these with some given in the literature.

\,

We say that the RGNEP satisfies the
coerciveness condition
$(C_0)$ at $\rho>0$ if:
\begin{enumerate}
\item $X\cap \overline{B}_\rho\neq\emptyset$  and
\item for each $x\in X$ with $\|x\|=\rho$, there exists $y\in X$ such that $\|y\|<\rho$ and
$\theta_\nu(y^\nu,x^{-\nu})\leq\theta_\nu(x)$,  for all $\nu\in N$.
\end{enumerate}


We now state our first existence result under coerciveness condition $(C_0)$ and convexity assumption. 

\begin{theorem}\label{TC0}
Assume that $X$ is a convex, closed and non-empty subset of $\R^n$, and each objective function is continuous. If the RGNEP satisfies the coerciveness condition $(C_0)$ and each objective function is convex with respect to its player's variable, then 
the set $\Rgnep(\{\theta_\nu, X\}_{\nu\in N})$ is non-empty.
\end{theorem}
\begin{proof} 
Consider the Nikaido-Isoda function $f:\R^n\times\R^n\to\R$ \cite{Nikaido-Isoda} defined as
\[
 f(x,y)=\sum_{\nu\in N}\theta_\nu(y^{\nu},x^{-\nu})-\theta_\nu(x)
\]
which is continuous in both arguments and convex with respect to its second one. Now, to show that the RGNEP admits at least one solution, it is enough to show that there exists a point $\hat{x}\in X$ such that
\begin{align}\label{EP}
 f(\hat{x},y)\geq0,\mbox{ for all }y\in X,
\end{align}
due to this point $\hat{x}$ being a generalized Nash equilibrium.
We notice that for $X_\rho= X\cap \overline{B_\rho}$, by the famous Fan's minimax theorem, there is a point $\hat{x}\in X_\rho$ such that
\begin{align}\label{EP2}
 f(\hat{x},y)\geq0,\mbox{ for all }y\in X_\rho.
\end{align}

If $\hat{x}$ does not verify \eqref{EP}, then there exists $z\in X$ such that $f(\hat{x},z)<0$. Now, if $\|\hat{x}\|<\rho$, then by convexity of
$f(\hat{x},\cdot)$ and the fact that $f(\hat{x},\hat{x})=0$ we have there exists $t\in]0,1[$ such that $t\hat{x}+(1-t)z\in X_\rho$ and
$f(\hat{x},t\hat{x}+(1-t)z)<0$, which is a contradiction with \eqref{EP2}.

If $\|\hat{x}\|=\rho$, the coerciveness condition $(C_0)$ implies that
 there is $y\in X$ such that $\|y\|<\rho$ and
$f(\hat{x},y)\leq0$. By convexity of $f(\hat{x},\cdot)$
there exists $t\in]0,1[$ such that $ty+(1-t)z\in X_\rho$ and
$f(\hat{x},ty+(1-t)z)<0$, and we again get a contradiction with \eqref{EP2}.
\end{proof}

We stablish below another existence result under semi strict quasi-convexity. 
We say that the RGNEP satisfies the coerciveness condition $(C_1)$ at $\rho>0$ if:
\begin{enumerate}
\item $X\cap \overline{B}_\rho\neq\emptyset$  and
\item  for each $x\in X$ with $\|x\|=\rho$, there exists $y\in \mathcal{X}(x)$ such that $(y^\nu,x^{-\nu})\in B_\rho$ and $\theta_\nu(y^\nu,x^{-\nu})\leq\theta_\nu(x),$ for all $\nu\in N$.
\end{enumerate}

\begin{theorem}\label{TC1}
Assume that $X$ is a convex, closed and non-empty subset of $\R^n$, and each objective function is pseudo-continuous. If the RGNEP satisfies the coerciveness condition $(C_1)$ and each objective function is semi strictly quasi-convex with respect to its player's variable;
then 
the set $\Rgnep(\{\theta_\nu, X\}_{\nu\in N})$ is non-empty.
\end{theorem}
\begin{proof}
By considering the RGNEP defined by $X_\rho=X\cap\overline{B}_{\rho}$ and the objective functions $\theta_\nu$. By Theorem \ref{Rosen-pseudocontinuous}, there exists $\hat{x}\in X_\rho$ such that
for each $\nu$
\[
\theta_\nu(\hat{x})\leq \theta_\nu(x^\nu,\hat{x}^{-\nu})\mbox{  for all }x^{\nu}\mbox{ such that }(x^\nu,\hat{x}^{-\nu})\in X_\rho.
\]
If $\hat{x}$ is not a generalized Nash equilibrium for the RGNEP associated $X$, then there exists $\nu_0$ and $x^{\nu_0}$ such that $(x^{\nu_0},\hat{x}^{-\nu_0})\in X$ and
$\theta_{\nu_0}(x^{\nu_0},\hat{x}^{-\nu_0})<\theta_{\nu_0}(\hat{x})$.
Thus, $x=(x^{\nu_0},\hat{x}^{-\nu_0})\notin X_\rho$.

Now, if $\|\hat{x}\|<\rho$, then by the semi-strictly quasi-convexity of $\theta_{\nu_0}(\cdot,\hat{x}^{-\nu_0})$ there exists $t\in]0,1[$ such that $z=t\hat{x}+(1-t)x\in X_\rho$  and
$\theta_{\nu_0}(z)<\theta_{\nu_0}(\hat{x})$,
which is a contradiction. Now, if $\|\hat{x}\|=\rho$,  by coerciveness condition $(C_1)$, there is $y\in\mathcal{X}(\hat{x})$ such that
$(y^{\nu_0},\hat{x}^{-\nu_0})\in B_\rho$ and $\theta_{\nu_0}(y^{\nu_0},\hat{x}^{-\nu_0})\leq \theta_{\nu_0}(\hat{x})$.
It is clear that the vector $w=(y^{\nu_0},\hat{x}^{-\nu_0})\in X$ and this implies the existence of some $t\in]0,1[$ such that  $z=tx+(1-t)w\in X_\rho$. By semi strict quasi-convexity of $\theta_{\nu_0}(\cdot,\hat{x}^{-\nu_0})$, we obtain
$\theta_{\nu_0}(z)<\theta_{\nu_0}(\hat{x})$ and we get a contradiction. Therefore, the proof is complete.
\end{proof}

We now give an existence result under  quasi-convexity assumption. 
We say that the RGNEP satisfies the coerciveness condition $(C_2)$ at $\rho>0$ if:

\begin{enumerate}
\item $X\cap \overline{B}_\rho\neq\emptyset$ and
\item  for each $x\in X\setminus  \overline{B}_{\rho}$, there exists $y\in \mathcal{X}(x)$ such that $(y^\nu,x^{-\nu})\in\overline{ B}_\rho$ and moreover $\theta_\nu(y^\nu,x^{-\nu})\leq\theta_\nu(x)$, for all $\nu\in N$.
\end{enumerate}

\begin{theorem}\label{CC2}
Assume that $X$ is a convex, closed and non-empty subset of $\R^n$, and each objective function is pseudo-continuous. If the RGNEP satisfies the coerciveness condition $(C_2)$ and each objective function is quasi-convex with respect to its player's variable;
then 
the set $\Rgnep(\{\theta_\nu, X\}_{\nu\in N})$ is non-empty.

\end{theorem}
\begin{proof}
Since $X$ is closed and nonempty, thanks to Lemma \ref{P1}, the coerciveness condition $(C_2)$ implies that  $X_\rho=X\cap\overline{B}_\rho$ is a compact, convex and non-empty subset of $\R^n$.  
By Theorem \ref{Rosen-pseudocontinuous}, there exists $\hat{x}\in\Rgnep(\{\theta_n,X_\rho\}_{\nu\in N})$, that is  $\hat{x}\in X_\rho$ such that
for each $\nu$
\[
\theta_\nu(\hat{x})\leq \theta_\nu(x^\nu,\hat{x}^{-\nu})\mbox{  for all }x^{\nu}\mbox{ such that }(x^\nu,\hat{x}^{-\nu})\in X_\rho.
\]
If $\hat{x}$ is not a generalized Nash equilibrium for the RGNEP associated $X$, then there exists $\nu_0$ and $x^{\nu_0}$ such that $(x^{\nu_0},\hat{x}^{-\nu_0})\in X$ and
\[
\theta_{\nu_0}(x^{\nu_0},\hat{x}^{-\nu_0})<\theta_{\nu_0}(\hat{x}).
\]
Thus $x=(x^{\nu_0},\hat{x}^{-\nu_0})\notin X_\rho$.
By the coerciveness condition $(C_2)$, there is $y\in \mathcal{X}(x)$ such that
\[
(y^{\nu_0},\hat{x}^{-\nu_0})\in \overline{B}_\rho\mbox{ and }\theta_{\nu_0}(y^{\nu_0},\hat{x}^{-\nu_0})\leq \theta_{\nu_0}(x).
\]
Since $y\in \mathcal{X}(x)$, we deduce that $(y^{\nu_0},\hat{x}^{-\nu_0})\in X\cap\overline{B}_\rho$. Thus,
\[
\theta_{\nu_0}(y^{\nu_0},\hat{x}^{-\nu_0})\leq \theta_{\nu_0}(x)<\theta_{\nu_0}(\hat{x}),
\]
which is a contradiction. Therefore, the proof is complete.
\end{proof}

\begin{remark}
Here we present some remarks about our coerciveness conditions:
\begin{enumerate}
\item First, we notice that $(C_2)$ implies $(C_1)$. Indeed, if $\rho>0$ is associated to $(C_2)$, then it is not complicated to verify that $(C_1)$ holds with $\rho_1=2\rho$.

\item Second, by considering the Euclidean norm in $\R^n$, for any $x,y\in\R^n$ such that $\|x\|=\rho$ and
$(y^{\nu},x^{-\nu})\in B_\rho$ for all $\nu\in N$,
we have $\|y\|<\rho.$ 
Thus $(C_1)$ implies $(C_0)$, if $X=K$. However, when $X\neq K$ this implication does not hold. Indeed,
consider  $X=\{(x^1,x^2)\in\R^2:~ x^1>0~\wedge~x^2\geq 1/x^1\}$, $\rho=2\sqrt{2}$, and the vectors $x=(2,2)$ and $y=(3/4,3/4)$. It is clear that $x\in X$, $y\in \mathcal{X}(x)$ and $  (2,3/4),(3/4,2)\in B_{2\sqrt{2}}$.
However, $y\notin X$, see Figure \ref{Ex2}.
\begin{figure}[h!]
\centering
\begin{tikzpicture}[scale=0.9,>=latex]
\fill[gray!35,samples=200,domain=0.2:5](0.2,5)--plot(\x,{1/\x})--(5,0.2)--(5,5)--(0.2,5);
\draw[domain=0.225:5,samples=200,line width=0.6,<->](0.2,5)--plot(\x,{1/\x})--(5,0.2);
\draw[domain=0:2.830,line width=0.7,samples=200]plot(\x,{sqrt(8-(\x)^2)});
\fill(2,2)circle(2pt)node[above right]{$x=(2,2)$};
\fill(3/4,3/4)circle(2pt)node at (0.5,0.7)[left]{$(3/4,3/4)=y$};
\draw[dotted,->](2,0.5)--(2,5);
\draw[dotted,->](0.5,2)--(5,2);
\draw[<->](0.5,5)--(0.5,0.5)--(5,0.5);
\draw[->](0,0)--(5,0)node[below right]{$x^1$};
\draw[->](0,0)--(0,5)node[left]{$x^2$};
\draw(3,0)node[below]{$B_\rho$};
\end{tikzpicture}
\caption{$y\in\mathcal{X}(x)$ but $y\notin X$. }\label{Ex2}
\end{figure}
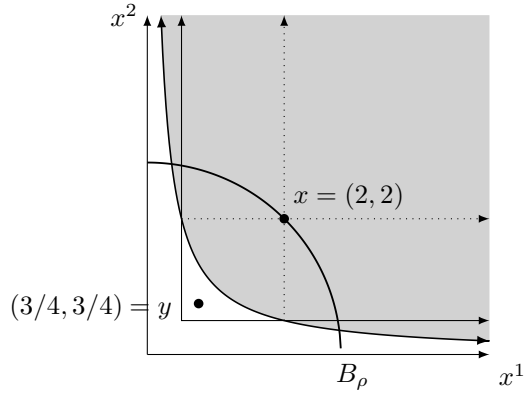

\item In the coerciveness condition in \cite{CHS-2020} it is assumed that $\mathcal{X}(x)\cap \overline{B}_\rho\neq\emptyset$, for all $x\in X\cap\overline{B}_\rho$. However, in our case we have $x\in\mathcal{X}(x)$ for every $x\in X$. Thus, this assumption is trivially satisfied.  Thanks to the above remark, we can adopt this coerciveness condition to our case as below: $(C_3)$ at $\rho>0$ if 
\begin{itemize}
\item $X\cap \overline{B}_\rho\neq\emptyset$, and
\item  for each $x\in X\setminus \overline{B}_\rho$, there exists $y\in \mathcal{X}(x)$ with $\|y\|<\rho$ and moreover $\theta_\nu(y^\nu,x^{-\nu})\leq\theta_\nu(x)$, for all $\nu\in N$.
\end{itemize}
Thus, $(C_3)$  is weaker than 
$(C_1)$. On the other hand, the following example says that Theorem \ref{TC0} is not a consequence from Theorem 4.6 in \cite{CHS-2020}. As a matter fact, consider $X$ as in the above remark,
and the functions $\theta_1,\theta_2:\R^2\to\R$ defined as
\[
\theta_1(x^1,x^2)=\theta_2(x^1,x^2)=(x^1)^2+(x^2)^2.
\]
Clearly,  for $x=(2\sqrt{2},\sqrt{2}/4)$ there is not $y\in \mathcal{X}(x)$ such that $\|y\|<2\sqrt{2}$. That means, this Nash game proposed by Rosen does not satisfy the coerciveness condition $(C_3)$. However, it satisfies $(C_0)$.

\item Regarding the above remark and the second part of the coerciveness condition in \cite{CZ-2018}, we notice that, if  $X$ is closed, convex and non-empty, and $\mathcal{X}(u)\cap\overline{B}_\rho\neq\emptyset$ for all $u\in X$; then the set $X\cap \overline{B}_\rho$ is non-empty. Indeed, assume that $X\cap \overline{B}_\rho=\emptyset$.
Let $u$ be an element of $X$ such that $\|u\|$ is minimum,
 and $z$ be an element of $\mathcal{X}(u)\cap\overline{B}_\rho$. Then $\|u\|>\rho\geq \|z\|$ and $(z^\nu,u^{-\nu})\in X$ for all $\nu$. It is clear that
\[
z=\sum_{\nu}(z^{\nu},u^{-\nu})-(p-1)u
\]
Thus,
 $tu+(1-t)z=(t-(1-t)(p-1))u+(1-t)\sum_{\nu}(z^{\nu},u^{-\nu})\in X$
for all $\dfrac{p-1}{p}<t\leq 1$. This implies $\|tu+(1-t)z\|\leq t\|u\|+(1-t)\|z\|<\|u\|$ and we get a contradiction. 

Now, the converse of this affirmation does not hold in general.  Indeed,
consider the set $X=\{(x^1,x^2)\in\R^2: x^2\geq x^1\geq0\}$ and $\rho=1$. It is clear that $X\cap \overline{B}_1\neq\emptyset$. However, $\mathcal{X}(2,2)=\{(x^1,x^2)\in\R^2:~0\leq x^1\leq 2\wedge x^2\geq 2\}$ and $\mathcal{X}(2,2)\cap \overline{B}_1=\emptyset$, see Figure \ref{Ex1}.
\begin{figure}[h!]
\centering
\begin{tikzpicture}[>=latex,scale=0.8]
\fill[gray!30](0,0)--(0,3)--(3,3)--(0,0);
\draw[fill=gray!20](0,0)circle(1);
\fill[gray!45](0,3)--(0,2)--(2,2)--(2,3)--(0,3);
\draw[line width=0.7](0,3)--(0,0)--(3,3);
\fill(2,2)circle(1.25pt)node[right]{$(2,2)$};
\draw(1,2.5)node[]{$\mathcal{X}(2,2)$};
\draw[line width=0.7](0,2)--(2,2)--(2,3);
\draw[->](-1.5,0)--(3,0)node[below]{$x^1$};
\draw[->](0,-1.5)--(0,3)node[left]{$x^2$};
\draw(0.85,-0.85)node[right]{$B_1$};
\end{tikzpicture}
\caption{$\mathcal{X}(u)\cap \overline{B}_\rho=\emptyset$.}\label{Ex1}
\end{figure}
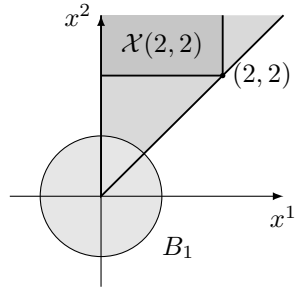

\item Finally,  Theorem \ref{CC2} is not a consequence of Theorem 3.2 in \cite{CZ-2018}, because any subset $K_\nu$ is not necessary closed. For instance, consider $X$ as in the second remark. We can see that $K_1=K_2=]0,+\infty[$. 
\end{enumerate}

\end{remark}

\section*{Conclusions}

The GNEP proposed by Rosen can be reduced to a classical Nash game, and as a consequence, we proved that Rosen's theorem follows from a classical result given by Arrow and Debreu. We also present an existence result for these kinds of games under quasi-convexity and pseudo-continuity. Moreover, sufficient and necessary conditions are established for a point to be a generalized Nash equilibrium using variational inequalities. Finally, we give some coerciveness conditions to guarantee the existence of generalized Nash equilibria on unbounded constraints.

\bibliographystyle{spmpsci}      

\begin{thebibliography}{10}

\bibitem{Arrow-Debreu}
K.~J. Arrow and G.~Debreu.
\newblock Existence of an equilibrium for a competitive economy.
\newblock {\em Econometrica}, 22(3):265--290, 1954.

\bibitem{JP-Aubin-1998}
J.~P. Aubin.
\newblock {\em Optima and equilibria : an introduction to nonlinear analysis /
  Jean-Pierre Aubin ; translated from the French by Stephen Wilson.}
\newblock Graduate texts in mathematics, 140. Springer, Berlin ;, 2nd ed.
  edition, 1998.

\bibitem{Aussel-book}
D.~Aussel.
\newblock {New developments in Quasiconvex optimization}.
\newblock In {\em {Fixed Point Theory, Variational Analysis, and
  Optimization}}, pages 173--208. Taylor \& Francis, 2014.

\bibitem{AC-2013}
D.~Aussel and J.~Cotrina.
\newblock {Quasimonotone quasivariational inequalities: existence results and
  applications}.
\newblock {\em J. Optim. Theory Appl.}, 158:637--652, 2013.

\bibitem{Aussel-Dutta}
D.~Aussel and J.~Dutta.
\newblock {Generalized Nash Equilibrium Problem, Variational Inequality and
  Quasiconvexity}.
\newblock {\em Oper. Res. Lett.}, 36(4):461--464, 2008.

\bibitem{AAD2014}
D.~Aussel and J.~Dutta.
\newblock Addendum to “generalized Nash equilibrium problem, variational
  inequality and quasiconvexity” [Oper. Res. Lett. 36 (4) (2008) 461–464].
\newblock {\em Oper. Res. Lett.}, 42(6):398, 2014.

\bibitem{AH04}
D.~Aussel and N.~Hadjisavvas.
\newblock {On Quasimonotone Variational Inequalities}.
\newblock {\em J. Optim. Theory Appl.}, 121(2):445--450, 2004.

\bibitem{AuHad2005}
D.~Aussel and N.~Hadjisavvas.
\newblock {Adjusted Sublevel Sets, Normal Operator, and Quasi-convex
  Programming}.
\newblock {\em SIAM Journal on Optimization}, 16(2):358--367, 2005.

\bibitem{AS-2017}
D.~Aussel and A.~Sultana.
\newblock {Quasi-variational inequality problems with non-compact valued
  constraint maps}.
\newblock {\em J. Math Anal. Appl.}, 456(2):1482--1494, 2017.

\bibitem{BCC21}
O.~Bueno, C.~Calder\'on, and J.~Cotrina.
\newblock {A note on coupled constraint Nash games}.
\newblock {\em Preprint on ArXiv: 2201.04262, pp. 11}, 2022.

\bibitem{BueCot16}
O.~Bueno and J.~Cotrina.
\newblock On maximality of quasimonotone operators.
\newblock {\em Set-Valued and Variational Analysis}, 27(1):87--101, 2019.

\bibitem{Cavazzuti}
E.~Cavazzuti, M.~Pappalardo, and M.~Passacantando.
\newblock Nash equilibria, variational inequalities, and dynamical systems.
\newblock {\em  J. Optim. Theory Appl.},
  114(3):491--506, 2002.

\bibitem{CK-2015}
J.~Contreras, J.~B. Krawczyk, and J.~Zuccollo.
\newblock Economics of collective monitoring: a study of environmentally
  constrained electricity generators.
\newblock {\em CMS}, 13(3):349--369, 2016.


\bibitem{CHS-2020}
J.~Cotrina, A.~Hantoute, and A.~Svensson.
\newblock Existence of quasi-equilibria on unbounded constraint sets.
\newblock {\em Optimization}, 0(0):1--18, 2020.

\bibitem{CZ-2018}
J.~{Cotrina} and J.~{Z{\'u}{\~n}iga}.
\newblock {Time-dependent generalized Nash equilibrium problem}.
\newblock {\em J. Optim. Theory Appl.}, 179:1054--1064, 2018.


\bibitem{Co21}
J.~Cotrina.
\newblock {Remarks on pseudo-continuity}.
\newblock {\em Minimax Theory and its Applications}, Accepted, 2022.

\bibitem{Facchinei2007}
F.~Facchinei and C.~Kanzow.
\newblock {Generalized Nash equilibrium problems}.
\newblock {\em 4OR}, 5(3):173--210, Sep 2007.

\bibitem{FACCHINEI2007159}
F.~Facchinei, A.~Fischer, and V.~Piccialli.
\newblock On generalized Nash games and variational inequalities.
\newblock {\em Oper. Res. Let}, 35(2):159 -- 164, 2007.

\bibitem{Fan}
K.~Fan.
\newblock Fixed-point and minimax theorems in locally convex topological linear
  spaces.
\newblock {\em Proc. Nat. Acad.  Sci. USA}, 38(2):121--126, 1952.

\bibitem{Glicksberg}
I.~L. Glicksberg.
\newblock A further generalization of the Kakutani fixed point theorem, with
  application to Nash equilibrium points.
\newblock {\em Proc.  Am. Math. Soc.},
  3(1):170--174, 1952.

\bibitem{HARKER199181}
P.~T. Harker.
\newblock {Generalized Nash games and quasi-variational inequalities}.
\newblock {\em Eur. J. Oper. Res.}, 54(1):81 -- 94, 1991.

\bibitem{John2001}
R.~John.
\newblock {A Note on Minty Variational Inequalities and Generalized
  Monotonicity}.
\newblock In N.~Hadjisavvas, J.~E. Mart{\'i}nez-Legaz, and J.-P. Penot,
  editors, {\em {Generalized Convexity and Generalized Monotonicity:
  Proceedings of the 6th International Symposium on Generalized
  Convexity/Monotonicity, Samos, September 1999}}, pages 240--246, Berlin,
  Heidelberg, 2001. Springer Berlin Heidelberg.

\bibitem{K-2005}
J.~B. Krawczyk.
\newblock Coupled constraint Nash equilibria in environmental games.
\newblock {\em Resour. Energy Econ.}, 27(2):157 -- 181, 2005.

\bibitem{K-2016}
J.~B. Krawczyk and M.~Tidball.
\newblock Economic problems with constraints: How efficiency relates to
  equilibrium.
\newblock {\em Int. Game Theory Rev.}, 18(04):1650011, 2016.

\bibitem{Piotr}
P.~Ma\'ckowiac.
\newblock {Some remarks on lower hemicontinuity of convex multivalued
  mappings}.
\newblock {\em Econ. Theory}, 28:227--233, 2006.

\bibitem{MORGAN2007}
J.~Morgan and V.~Scalzo.
\newblock Pseudocontinuous functions and existence of Nash equilibria.
\newblock {\em J. Math. Econ.}, 43(2):174--183, 2007.

\bibitem{Nash}
J.~Nash.
\newblock Non-cooperative games.
\newblock {\em Ann. Math.}, 54(2):286--295, 1951.

\bibitem{Nikaido-Isoda}
H.~Nikaid{\^o} and K.~Isoda.
\newblock {Note on noncooperative convex games}.
\newblock {\em Pacific J. Math.}, 52:807--815, 1955.

\bibitem{Ponstein}
J.~Ponstein.
\newblock Existence of equilibrium points in non-product spaces.
\newblock {\em SIAM Journal on Applied Mathematics}, 14(1):181--190, 1966.


\bibitem{Rosen}
J.~B. Rosen.
\newblock Existence and uniqueness of equilibrium points for concave n-person
  games.
\newblock {\em Econometrica}, 33(3):520--534, 1965.

\bibitem{Scalzo}
V.~Scalzo.
\newblock Pseudocontinuity is necessary and sufficient for order-preserving
  continuous representations.
\newblock {\em Real Anal. Exchange}, 34(1):239--248, 2009.

\bibitem{Vardar2019}
B.~Vardar and G.~Zaccour.
\newblock Strategic bilateral exchange of a bad.
\newblock {\em Oper. Res. Lett.}, 47(4):235 -- 240, 2019.

\bibitem{Yu-Fa-Yang2016}
J.~Yu, N.-F. Wang, and Z.~Yang.
\newblock {Equivalence results between Nash equilibrium theorem and some fixed
  point theorems}.
\newblock {\em Fixed Point Theory Appl}, 2016(1):69, 2016.

\end{thebibliography}

\end{document}